\documentclass[12pt]{amsart}
\usepackage[run=2]{auto-pst-pdf} 
\usepackage{hyperref}
\usepackage{color}

\usepackage{graphicx}
\usepackage{xypic}

\usepackage{amsmath,amsfonts,amssymb}

\usepackage{pstricks}
\usepackage{pst-plot}

\theoremstyle{plain}
\newtheorem*{theorem*}{Theorem}
\newtheorem*{lemma*} {Lemma}
\newtheorem*{corollary*} {Corollary}
\newtheorem*{proposition*} {Proposition}
\newtheorem*{conjecture*}{Conjecture}

\newtheorem{theorem}{Theorem}[section]
\newtheorem{lemma}[theorem]{Lemma}
\newtheorem{corollary}[theorem]{Corollary}
\newtheorem{proposition}[theorem]{Proposition}

\newtheorem{question}[theorem]{Question}

\theoremstyle{definition}

\theoremstyle{remark}
\newtheorem*{remark}{Remark}
\newtheorem*{definition}{Definition}

\newtheorem*{claim}{Claim}

\theoremstyle{definition}

\def\Q{\Bbb{Q}}

\def\Z{\Bbb{Z}}
\def\R{\Bbb{R}}
\def\C{\Bbb{C}}
\def\N{\Bbb{N}}
\def\be{\begin{equation}} \def\ee{\end{equation}}
\def\id{\operatorname{id}}

\def\eps{\epsilon}
\def\sign{\operatorname{sign}}
\def\null{\operatorname{null}}

\def\l{\lambda}

\def\sm{\setminus}
\def\bp{\begin{pmatrix}}
\def\ep{\end{pmatrix}}
\def\bn{\begin{enumerate}}
\def\en{\end{enumerate}}

\def\ba{\begin{array}}
\def\ea{\end{array}}

\def\L{\Lambda}

\def\tpm{[t^{\pm 1}]}

\def\s{\sigma}

\def\fr12{\frac{1}{2}}
\def\diag{\operatorname{diag}}

\def\ol{\overline}
\def\Im{\operatorname{Im}}

\def\S{\Sigma}

\def\vare{\varepsilon}

\def\PP{\mathcal{P}}
\def\co{\colon}
\def\gammamax{\gamma_{max}}
\def\gammamin{\gamma_{min}}

\def\rt{R[t^{\pm 1}]}
\def\realt{\R[t^{\pm 1}]}
\def\ct{\C[t^{\pm 1}]}

\def\zt{\Z[t^{\pm 1}]}
\def\qt{\Q[t^{\pm 1}]}

\def\hom{\operatorname{Hom}}

\def\wti{\widetilde}

\def\nr{n_{\mathbb{R}}}

\def\to{\mathchoice{\longrightarrow}{\rightarrow}{\rightarrow}{\rightarrow}}
\makeatletter
\newcommand{\shortxra}[2][]{\ext@arrow 0359\rightarrowfill@{#1}{#2}}
\def\longrightarrowfill@{\arrowfill@\relbar\relbar\longrightarrow}
\newcommand{\longxra}[2][]{\ext@arrow 0359\longrightarrowfill@{#1}{#2}}

\makeatother

\begin{document}

\title{The unknotting number and classical invariants II}

\date{03 July 2012}

\author{Maciej Borodzik}
\address{Institute of Mathematics, University of Warsaw, Warsaw, Poland}
\email{mcboro@mimuw.edu.pl}

\author{Stefan Friedl}
\address{Mathematisches Institut\\ Universit\"at zu K\"oln\\   Germany}
\email{sfriedl@gmail.com}

\thanks{The first author is supported by Polish MNiSzW Grant No N N201 397937 and also by the Foundation for Polish Science FNP}

\def\subjclassname{\textup{2000} Mathematics Subject Classification}
\expandafter\let\csname subjclassname@1991\endcsname=\subjclassname
\expandafter\let\csname subjclassname@2000\endcsname=\subjclassname

\subjclass{Primary 57M27}
\keywords{Levine--Tristram signatures, Blanchfield form}

\begin{abstract}
In \cite{BF12} the authors associated to a knot $K\subset S^3$ an invariant $n_\R(K)$ which is defined using the Blanchfield form and which gives a lower bound on the unknotting number.
In this paper we express  $n_\R(K)$ in terms of  Levine--Tristram signatures and nullities of $K$.
In the proof we also show that the Blanchfield form for any knot $K$ is diagonalizable over $\realt$.
\end{abstract}

\maketitle

\section{Introduction}

Let $K\subset S^3$ be a knot. Throughout this paper we assume that all knots are oriented. 
We denote by $X(K)=S^3\sm\nu K$ its  exterior. The
Blanchfield form (see \cite{Bl57} and also Section \ref{section:blanchfield} below) is the linking form on the Alexander module $H_1(X(K),\zt)$, i.e. a non-singular hermitian form
\[ \l(K)\colon H_1(X(K),\zt)\times H_1(X(K),\zt)\to \Q(t)/\zt.\]
In \cite{BF12} we denoted by  $n(K)$  the minimal size of a hermitian matrix $A(t)$ over $\zt$,
which represents the Blanchfield form and such that $A(1)$ is diagonalizable
over $\Z$. We then showed that $n(K)$ is a lower bound on the unknotting number $u(K)$.
Unfortunately $n(K)$ is  in general hard to compute. The weaker invariant
$n_\R(K)$ is the minimal size of a square matrix over $\realt$, which represents the Blanchfield form over $H_1(X(K);\realt)$.
In this paper we will show that $n_\R(K)$ is determined by the Levine-Tristram signatures and the nullities of $K$.

Before we state the main result of the paper, let us recall
that for a knot $K$ with Seifert matrix $V$ and $z\in S^1$ the Levine-Tristram signature \cite{Le69,Tr69} is defined as
\[\s_K(z)=\sign(V(1-z)+V^t(1-z^{-1})).\]
Furthermore, for $z\in \C\sm \{0,1\}$ the nullity is defined as
\[ \eta_K(z)=\null(V(1-z)+V^t(1-z^{-1})),\]
and we extend this definition to $z=1$ by setting $\eta_K(1)=0$. It is well-known that these definitions do not depend on the choice of a Seifert matrix.
These invariants now give rise to the following two invariants which will play a prominent role in this paper:
\[
\ba{rl}
\mu(K)&:=\frac{1}{2}\left(\max\{ \eta_K(z)+\s_K(z) \, |\, z\in S^1\}+\max\{ \eta_K(z)-\s_K(z) \, |\, z\in S^1\}\right)\\
\eta(K)&:=\max\{ \eta_K(z) \, |\, z\in\C\sm \{0\}\}.
\ea
\]
It is  straightforward, see Lemma \ref{lem:easy}, to show that $\mu(K)$ and $\eta(K)$ are lower bounds on $n_\R(K)$.
Our main theorem is now the following result, first announced in \cite{BF12}, which says that $n_\R(K)$ is in fact determined by $\mu(K)$ and $\eta(K)$.

\begin{theorem}\label{thm3}\label{mainthm}
For any knot $K$ we have
\[ n_\R(K)=\max\{\mu(K),\eta(K)\}.\]
\end{theorem}

\begin{remark}
\bn
\item
Since $V(1-z)+V^t(1-z^{-1})=(Vz-V^t)(z^{-1}-1)$ and $\Delta_K(z)=\det(Vz-V^t)$
it follows that  $\eta(K)$ is determined by the values of $\eta_K$ at the set of zeros of $\Delta_K(t)$.
Similarly we will  show (see Proposition \ref{prop:maxoncircle}) that $\mu(K)$ is  determined by the values of $\s_K$ and $\eta_K$
at the zeros of $\Delta_K(t)$ on the unit circle.
\item We denote by  $W(\Q(t))$  the Witt group of hermitian non-singular forms $\Q(t)^r\times \Q(t)^r\to \Q(t)$.
 Livingston  \cite{Li11} introduced the knot invariant
\[ \rho(K):=\ba{c}\mbox{minimal size of a hermitian matrix $A(t)$}\\
\mbox{representing $(1-t)V_K+(1-t^{-1})V_K^t$ in $W(\Q(t))$}\ea \]
and showed that it is a lower bound on the 4-genus. Furthermore Livingston showed that $\rho(K)$ can be determined using the Levine-Tristram signature function. This result is related in spirit to our result that $n_\R(K)$ is a lower bound on the unknotting number and that $n_\R(K)$ can be determined using Levine-Tristram signatures and nullities.
\en
\end{remark}

There are two main ingredients in the proof of Theorem \ref{mainthm}. The first one is that the Blanchfield form over $\R$ can be represented
by a diagonal matrix, see Section~\ref{section:diag}. The other one is the Decomposition Theorem
in Section~\ref{section:decomp} which is used twice in the proof of Theorem \ref{thm3}. More precisely, we first  show that the Blanchfield
form can be represented by an elementary diagonal matrix $E$ in Section~\ref{section:basicdiagonal}. We then use the Decomposition Theorem
to carefully rearrange terms on the diagonal of $E$ so as to decrease its size to exactly $n_\R$. This is done in Section~\ref{section:conclusion}.

To conclude the introduction we point out  that passing from a matrix $A(t)$ representing the Blanchfield form to an elementary diagonal matrix $E(t)$
(see Section~\ref{section:basicdiagonal}) is closely related to the classification of isometric structures over $\R$
done in \cite{Mi69} (see also \cite{Neu82,Ne95}). 
 For example, there is a one-to-one correspondence between indecomposable parts
of isometric structures over $\R$ and polynomials occurring on the diagonal of $E(t)$. We refer to \cite{BN13} for other applications
of this classification in  knot theory.


\subsection*{Acknowledgment.}
We wish to think Andrew Ranicki and Alexander Stoimenow for helpful conversations.
We are especially grateful to the referee for very carefully reading an earlier version of this paper.
The authors also would like to thank the Renyi Institute of Mathematics
for hospitality.

\section{Blanchfield forms}\label{sec:gbf}

In this section we review some definitions and results from our earlier paper \cite{BF12}.

\subsection{Blanchfield forms}\label{section:tbp}\label{section:blanchfield}

Let $R\subset \R$ be a subring. We denote by
\[ p(t)\mapsto \ol{p(t)}:=p(t^{-1})\]
 the usual involution on $\rt$.
Throughout the paper we will denote the quotient field of $\rt$ by $Q_R(t)$.
The involution on $\rt$ extends in a canonical way to an involution on $Q_R(t)$.
We will henceforth always view $\rt$ and $Q_R(t)$ as rings with involution.
Given an $\rt$-module $M$ we will denote by $\ol{M}$ the module with the `involuted' $\rt$-module structure,
i.e. given $p(t)\in\rt$ and $m$ in the abelian group $\ol{M}=M$ we define the $\rt$-module structure in $\ol{M}$ by $p(t)\cdot m:=p(t^{-1})\cdot m$,
where the multiplication on the right hand side is given by the multiplication in the $\rt$-module $M$.
\medskip

A \emph{Blanchfield form} over $\rt$ is a hermitian non-singular  form
\[\lambda\colon H\times H\to Q_R(t)/\rt,\]
where $H$ is a finitely generated torsion $\rt$-module.
Recall, that a form is called \emph{hermitian}
if
\[ \l(a,p_1b_1+p_2b_2)=\l(a,b_1)p_1+\l(a,b_2)p_2 \mbox{ for any }a,b_1,b_2\in H\mbox{ and any }p_1,p_2\in \rt,\]
and if
\[ \l(a_1,a_2)=\ol{\l(a_2,a_1)}\mbox{ for any }a_1,a_2\in H.\]
Also, a form is called  \emph{non-singular} if the map
\[ \ba{rcl} H&\to& \ol{\hom(H,Q_R(t)/\rt)} \\
a&\mapsto & \l(a,b) \ea \]
is an isomorphism.

\subsection{Blanchfield forms and hermitian matrices}\label{section:class}
Let $R\subset \R$ be a subring.
 Given a hermitian $n\times n$ matrix $A$ over $\rt$ with $\det(A)\ne 0$ we denote by
$\lambda(A)$ the form
\[ \ba{rcl} \rt^n/A\rt^n \times \rt^n/A\rt^n&\to & Q_R(t)/\rt \\
(a,b) &\mapsto & \ol{a}^t A^{-1} b,\ea \]
where we view $a,b$ as represented by column vectors in $\rt^n$. Note that $\lambda(A)$ is a Blanchfield form, i.e. it is a hermitian, non-singular
form. The following result due to Ranicki \cite{Ra81} says that in fact all Blanchfield forms are isomorphic to some $\lambda(A)$.

\begin{proposition}
Let $R\subset \R$ be a subring. Given a Blanchfield form $\l$ over $\rt$ there exists a hermitian matrix $A$ with $\det(A)\ne 0$ such that 
$\l\cong \l(A)$. 
\end{proposition} 

\begin{proof}
The proposition is an immediate consequence of \cite[Proposition~1.7.1]{Ra81} together with \cite[Proposition~3.4.3]{Ra81}.
\end{proof}

We will often appeal to the following result which also follows from work of  Ranicki's \cite{Ra81}:

\begin{proposition} \label{prop:ra81}
Let $R\subset \R$ be a subring and let $A$ and $B$ be hermitian matrices over $\rt$ such that  $\det(A)\ne 0$ and $\det(B)\ne 0$.
Then $\lambda(A)$ and $\lambda(B)$ are isometric forms  if and only if $A$ and $B$ are related by a sequence of  the following three moves:
\bn
\item replace $C$ by $PC\ol{P}^t$ where $P$ is a matrix over $\rt$ with $\det(P)$ a unit in $\rt$,
\item replace $C$ by the block sum $C\oplus D$ where $D$ is a hermitian matrix over $\rt$ such that  $\det(D)$ is a unit in $\rt$,
\item the inverse of $(2)$.
\en
\end{proposition}

\begin{proof}
The `if' direction of the proposition is elementary, whereas the `only if' direction is an immediate consequence of results in \cite{Ra81}.
Since the language in \cite{Ra81} is somewhat different we will quickly outline how the proposition follows from Ranicki's result.

We first note that any hermitian $n\times n$-matrix $C$ over $\rt$ with $\det(C)\ne 0$  defines a symmetric non-degenerate hermitian form 
\[ \ba{rcl} \Phi(C)\co \rt^n\times \rt^n&\mapsto & \rt \\
(v,w)&\mapsto& \ol{v}^t Cw.\ea \]
It is clear that if $C$ and $D$ are two matrices with non-zero determinants, then  $\Phi(C)$ and $\Phi(D)$ are isometric if and only if there exists a matrix $P$ over $\rt$ such that $\det(P)$ is a unit and such that $C=PD\ol{P}^t$. 

Now suppose that  $A$ and $B$ are two hermitian matrices over $\rt$ such that  $\det(A)\ne 0$ and $\det(B)\ne 0$
 and such that  $\lambda(A)\cong \lambda(B)$.
It then follows from  \cite[Proposition~1.7.1]{Ra81} together with \cite[Proposition~3.4.3]{Ra81}
that there exist hermitian matrices $X$ and $Y$ over $\rt$ such that $\det(X)$ and $\det(Y)$ are units and 
such that  $\Phi(A\oplus X)$ and $\Phi(B\oplus Y)$ are isometric.
The proposition follows from the above observation on isometric hermitian forms.

\end{proof}

\subsection{Definition of the Blanchfield Form dimension $n_R(\l)$.}

We are now ready to give the key definition of this paper.
Let $R\subset \R$ be a subring. Given a Blanchfield form $\l$ over $\rt$ we define
\[ n_R(\l):=\begin{tabular}{c}\mbox{minimal size of a hermitian matrix $A$ over $\rt$ with $\l(A)\cong \l$}\\
\mbox{and such that $A(1)$ is diagonalizable over $R$.}\end{tabular}\]
If no such matrix $A$ exists, then we write $n_R(\l):=\infty$.

Note that given a hermitan matrix $A(t)$ over $\rt$ the matrix $A(1)$ is symmetric.
Since any symmetric matrix over $\R$ is diagonalizable it now follows that 
for a Blanchfield form $\l$ over $\realt$ we have 
\[ n_{\R}(\l):=\mbox{minimal size of a hermitian matrix $A$ over $\realt$ with $\l(A)\cong \l$}.\]

\subsection{The Blanchfield forms of a knot in $S^3$}

Now let $K\subset S^3$ be a knot and let $R\subset \R$ be a subring.  We  consider the following sequence of maps:
\be\label{equ:defbl}
\ba{rcl}
\Phi\colon H_1(X(K);\rt)&\to& H_1(X(K),\partial X(K);\rt)\\[2mm]
&\to& \ol{H^2(X(K);\rt)}\xleftarrow{\cong} \ol{H^1(X(K);Q_R(t)/\rt)}\\[2mm]
&\to &\ol{\hom_{\rt}(H_1(X(K);\rt),Q_R(t)/\rt)}.\ea \ee
Here the first map is the inclusion induced map, the second map is Poincar\'e duality,
the third map comes from the long exact sequence in cohomology corresponding to the coefficients
$0\to \rt\to Q_R(t)\to Q_R(t)/\rt\to 0$, and the last map is the evaluation map. All these maps are isomorphisms and we thus obtain
 a non-singular form
\[ \ba{rcl} \lambda_R(K)\colon  H_1(X(K);\rt)\times H_1(X(K);\rt)&\to& Q_R(t)/\rt\\
(a,b)&\mapsto & \Phi(a)(b),\ea \]
called the \emph{Blanchfield form of $K$ with $\rt$-coefficients}.
It is well-known, see e.g. \cite{Bl57}, that this form is hermitian, i.e. it is in fact a Blanchfield form in the above sense.
If $R=\Z$ then we just write $\l(K)=\l_{\Z}(K)$ and refer to $\l(K)$ as the Blanchfield form of $K$.
\medskip

We now turn to the study of the Blanchfield Form dimension of $\l(K)$.  Given a  subring $R\subset \R$ we define
\[ n_R(K):=n_R(\l_R(K)).\]
Furthermore, if $\S$ is  a Seifert surface of genus $g$ for $K$ then we can  pick  a basis of $H_1(\S;\Z)$ such that the corresponding Seifert matrix $V$ has the property that 
$V-V^t=\bp 0 & \id_g\\ -\id_g & 0 \ep$. Then the hermitian matrix
\begin{multline}\label{equ:akt}
A(t)=\bp (1-t^{-1})^{-1}\id_g &0\\ 0&\id_g \ep V \bp \id_g&0\\0&(1-t)\id_g\ep +\\
+\bp \id_g &0 \\ 0&(1-t^{-1})\id_g \ep V^t\bp (1-t)^{-1}\id_g &0 \\ 0&\id_g
\ep\end{multline}
has the property that $\l_R(A(t))\cong \l_R(K)$. 
Note that $A(1)$ is not necessarily diagonalizable, but the diagonal sum $A(1)\oplus (1)$ or $A(1)\oplus (-1)$ is non-singular
and indefinite, hence diagonalizable. We thus see that $n_R(K)\leq 2g+1$. See \cite[Section~4]{Ko89} or \cite[Section 2.2]{BF12} for the details.

\begin{remark}
Following H. Murakami \cite{Mu90} we consider   the algebraic  unknotting number $u_a(K)$ of $K$,
i.e. the minimal number of crossing changes needed to turn $K$ into a knot with trivial Alexander polynomial.
We then have the following inequalities
\[ n_R(K)\leq n_{\Z}(K)=u_a(K)\leq u(K).\]
The first and the last inequality follow almost immediately from the definitions.
The fact that $n_\Z(K)\leq u_a(K)$ was proved in \cite{BF12}.
The converse inequality is shown in \cite{BF13}.
\end{remark}

\subsection{Statement of the main theorem}\label{section:statementmainthm}

 Given a hermitian matrix $A$ over $\rt$ and $z\in S^1$ we  define
\[ \s_A(z):= \sign(A(z))-\sign(A(1))\]
and given any $z\in \C\sm \{0\}$ we define
\[ \eta_A(z):=\null(A(z)).\]
Furthermore we  define
\[
\ba{rl}
\mu(A)&=\frac{1}{2}\left(\max\{ \eta_A(z)+\s_A(z) \, |\, z\in S^1\}+\max\{ \eta_A(z)-\s_A(z) \, |\, z\in S^1\}\right)\\
\eta(A)&=\max\{ \eta_A(z) \, |\, z\in\C\sm \{0\}\}.
\ea
\]
We can now formulate the following corollary to Proposition \ref{prop:ra81}.

\begin{corollary}\label{cor:ltdefined}
Let $R\subset \R$ be a subring and let $A$ and $B$ be hermitian matrices over $\rt$ such that  $\det(A(1))$ and $\det(B(1))$ are non-zero.
If $\lambda(A)$ and $\lambda(B)$ are isometric, then for any $z\in S^1$  we have
\[ \s_A(z)=\s_B(z)\]
and for any $z\in \C\sm \{0\}$ we have
\[ \eta_A(z)=\eta_B(z). \]
\end{corollary}

\begin{proof}
The first claim concerning nullity is an immediate consequence of Proposition \ref{prop:ra81}.
We now turn to the proof of the claim regarding signatures.
First suppose that $B=PA\ol{P}^t$ where  $P$ is a matrix over $\rt$ such that  $\det(P)$ is a unit in $\rt$,
i.e. $\det(P)=rt^i$ for some $r\ne 0\in R$ and $i\in \Z$. Note that $\det(P(z))\ne 0$ for any $z$. We calculate
\[  \ba{rcl} \s_B(z)&=&\sign(B(z))-\sign(B(1))\\
&=& \sign(P(z)A(z)\ol{P(z)}^t)-\sign(P(1)A(1)P(1)^t)\\
&=&\sign(A(z))-\sign(A(1))\\
&=&\s_A(z).\ea \]
Now suppose that $B=A\oplus D$ where $D$ is a hermitian matrix over $\rt$ such that $\det(D)$ is a unit in $\rt$.
It is well-known that for any hermitian matrix $M$ over $\rt$ the map
\[ \ba{rcl} S^1&\to& \Z \\ z&\mapsto &\sign(M(z)) \ea \]
is continuous on $\{ z\in S^1 \, |\, \det(M(z))\ne 0\}$.
Since $\det(D(z))=\det(D)(z)\ne 0$ for any $z$ we see that $\sign(D(z))=\sign(D(1))$ for any $z$.
It now follows immediately that
\[ \s_A(z)=\sign(A(z))-\sign(A(1))=\sign(B(z))-\sign(B(1))=\s_B(z).\]
The corollary is now an immediate consequence of Proposition \ref{prop:ra81}.
\end{proof}

Let $\l\colon H\times H\to \R(t)/\realt$ be a Blanchfield form over $\realt$ such that multiplication by $t\pm 1$ is an isomorphism of $H$.
Let $B$ be a matrix over $\realt$ which represents $\l$. Note that the hypothesis on $H$ implies that $B(\pm 1)\ne 0$.
Given $z\in S^1$ we define $\s_\l(z)=\s_B(z)$ and given $z\in \C\sm\{0\}$ we define $\eta_\l(z)=\eta_B(z)$.
We furthermore define $\mu(\l)=\mu(B)$ and $\eta(\l)=\eta(B)$.
It follows from  Corollary \ref{cor:ltdefined} that these invariants do not depend on the choice of $B$.

We can now prove the easy direction of our main theorem.

\begin{lemma} \label{lem:easy}
Let $\l\colon H\times H\to \R(t)/\realt$ be a  Blanchfield form over $\realt$ such that multiplication by $t\pm 1$ is an isomorphism.
Then
\[ n_{\R}(\l)\geq \max\{\mu(\l),\eta(\l)\}.\]
\end{lemma}

\begin{proof}
Let $B$ be a hermitian matrix over $\L$
of size $n:=n_\R(\l)$ which represents $\l$. Of course for any $z\in\C\sm\{0\}$ we have $\null(B(z))\le n_\R(\l)$, in particular
\[\eta(\l)\le n_\R(\l).\]

To show that $\mu(\l)\le n_\R(\l)$ let us assume that $\sign(B(1))=a\in[-n,n]$.
Note that for any  $z\in S^1$ we have  $\pm\sign(B(z))+\null(B(z))\in[-n,n]$. It thus  follows
that $\eta_\l(z)+\s_\l(z)\in[-n-a,n-a]$ and $\eta_\l(z)-\s_\l(z)\in[-n+a,n+a]$. We infer that $\mu_\l(z)\le\frac12((n-a)+(n+a))=n$.
\end{proof}

Our main theorem now says that under a slight extra assumption the inequality in Lemma \ref{lem:easy} is in fact an equality.
More precisely, the goal of this paper is to prove the following theorem.

\begin{theorem}\label{main}
Let $\l\colon H\times H\to \R(t)/\realt$ be a  Blanchfield form over $\realt$ such that
multiplication by $t\pm 1$ is an isomorphism of $H$.
Then
\[ n_{\R}(\l)= \max\{\mu(\l),\eta(\l)\}.\]
\end{theorem}

The proof of this theorem will require all  of Sections \ref{section:lemmas} and  \ref{section:mainproof}.
Assuming Theorem \ref{main} we can now easily provide the proof of Theorem \ref{thm3}.

\begin{proof}[Proof of Theorem \ref{thm3}]
Let $K\subset S^3$ be a knot. We write $\l=\l_\R(K)$.
In  \cite[Section 3.1]{BF12} we showed that given any $z\in S^1$ we have $\s_\l(z)=\l_K(z)$ and given any $z\in \C\sm \{0\}$ we have
$\eta_\l(z)=\eta_K(z)$. In particular we have  $\mu(\l)=\mu(K)$ and $\eta(\l)=\eta(K)$.

It is well-known that the Alexander polynomial $\Delta_K$ of $K$ satisfies
$\Delta_K(1)=\pm 1$. Since $\Delta_K(-1)\equiv \Delta_K(1)=1 \bmod 2$ we deduce that  $\Delta_K(-1)\ne 0$.
It now follows easily that multiplication by $t-1$ and $t+1$ are isomorphisms of the Alexander module $H_1(X(K);\realt)$.

The theorem now follows immediately from the above observations and from Theorem \ref{main}.
\end{proof}


\section{Technical lemmas}\label{section:lemmas}

From now on  we write \[ \L:=\R\tpm \mbox{ and } \Omega:=\R(t).\]
Moreover, we write $S^1_+$ for the set of all points on $S^1$ with non-negative imaginary part.
Given $z_1,z_2\in S^1_+$ we can write $z_i=e^{2\pi it_i}$ for a unique $t_i\in [0,\pi]$.
We then write $z_1>z_2$ if $t_1>t_2$. Given $a,b\in S^1_+$ we use the usual interval notation to define subsets $[a,b), (a,b)$ etc. of $S^1_+$.

\subsection{Palindromic polynomials and elementary palindromic polynomials}\label{section:basicspols}

Let us recall the following well known definition.

\begin{definition}
An element $p\in\L$ is called \emph{palindromic} if $p(t)=p(t^{-1})$ as polynomials.
\end{definition}

We say that a function $f\co S^1\to \C$ is \emph{symmetric} if $f(z)=f(\ol{z})$ for all $z\in S^1$.
Note that if  $p$ is palindromic, then $z\mapsto p(z)$ is symmetric.  The next result will be used in the proof of Lemma~\ref{lem:aprox} below.

\begin{lemma}\label{lem:converge}
Palindromic polynomials form a dense subset in the space of all real-valued continuous symmetric functions on $S^1$.
\end{lemma}

In the lemma we mean `dense' with respect to the  supremum norm.

\begin{proof}
Note that each symmetric function  is determined by its values on $S^1_+$. On $S^1_+$, the palindromic
polynomials form a real algebra which separates points (note that they do not separate points on the whole $S^1$). By the Stone-Weierstrass
theorem (see e.g. \cite[Theorem 7.32]{Rud76}), for any real-valued continuous
function $f$, there
exists a sequence $p_n$ of palindromic polynomials converging to $f$ uniformly on $S^1_+$. As $f(\ol{z})=f(z)$ and $p_n(\ol{z})=p_n(z)$ for all $z\in S^1$,
this convergence extends to the convergence on $S^1$.
\end{proof}

We will make use of  the following terminology.
We write
\[ \Xi:=\{ \xi \in \C\sm \{0\}\, |\, \Im \xi\geq 0 \mbox{ and } |\xi|\leq 1 \}.\]
Furthermore, given $\xi\in \Xi$ we  define
\be\label{eq:defbasic}
B_\xi(t)=
\begin{cases}
t-\xi)(t-\ol{\xi})&\text{if $|\xi|=1$},\\
(t-\xi)(t-\ol{\xi})(1-t^{-1}\ol{\xi}^{-1})(1-t^{-1}\xi^{-1})&\text{if $|\xi|<1$ and $\xi\not\in\R$},\\
(t-\xi)(1-\xi^{-1}t^{-1})&\text{if $\xi\in\R\setminus{\pm 1}$}.
\end{cases}
\ee
These polynomials are  called the \emph{elementary palindromic} polynomials.
We conclude this section with the following observations:
\bn
\item  For any $\xi\in \Xi$ the polynomial $B_\xi(t)$ is a real, palindromic and monic polynomial.
\item For any $\xi$ we have $B_\xi(1)>0$, furthermore if $|\xi|<1$, then $B_\xi$ has no zeros on $S^1$, i.e. $B_\xi$ is positive on $S^1$.
\item Given any $z\in \C\sm \{0\}$ there exists a unique $\xi\in \Xi$ such that $z$ is a zero of $B_\xi(t)$.
Furthermore $B_\xi(t)$ is the unique real, palindromic, monic polynomial of minimal degree which has a zero at $z$.
\item Any palindromic
polynomial in $\L$ factors uniquely as the product of elementary palindromic polynomials and a constant in $\R$.
\en

\subsection{First results}

\begin{lemma}\label{l:step2}\label{lem:norm}
Let $P\in\L$ be palindromic. Then there exists $U\in \L$ with $P=U\ol{U}$ if and only if  $P(z)\ge 0$ for every $z\in S^1_+$.
\end{lemma}

\begin{proof}
If $P=U\ol{U}$ for some $U\in \L$, then for each $z\in S^1_+$ we obviously have $P\ge 0$. Conversely, assume that $P(z) \ge 0$ for every $z\in S^1_+$.
Note that by the above discussion this implies that  $P(z) \ge 0$ on $S^1$.
We proceed by induction on the
degree of $P$ (that is the number of zeros counted with multiplicities). If $P$ has degree $0$, then $P$ is constant and there is nothing to prove.

If $P$ has positive degree, we choose $\theta$ to be a zero of $P(t)$.
Since $P(t)=P(t^{-1})$ we see that $\theta^{-1}$ is also a zero of $P$.
Furthermore, since $P(t)$ is a real polynomial we see that if $\mu$ is a zero, then $\ol{\mu}$ is also a zero.
Thus, if $\theta$ is a zero, then $\theta,\ol{\theta},\theta^{-1},\ol{\theta}^{-1}$ are all zeros.

We first consider the case, that $\theta$ lies in $\theta\in \mathbb{C}\setminus S^1$ and that $\theta\not\in \R$.
 Let $\xi\in \Xi$ be the unique element such that $\theta$ is a zero of the elementary palindromic
polynomial $B_\xi(t)$.
Note that $B_\xi(t)$ divides $P(t)$. Furthermore note that
$P_2=\frac{P(t)}{B_\xi(t)}$  has smaller degree and is non-negative on $S^1$. By induction we have $P_2=U_2\ol{U_2}$.
The polynomial  $U=(t-\theta)(t-\ol{\theta})U_2$ then satisfies $P=U\ol{U}$.

We now consider the case that $\theta\in \R$ with $\theta \ne \pm 1$. The argument is almost identical to the first case except
that now, using the above notation, the polynomial  $U=(t-\theta)U_2$ has the desired properties.

We then consider the case that $\theta\in S^1$ with $\theta \ne \pm 1$.  As $P\ge 0$ on $S^1$, the order of the root of $P$ at $\theta$
must be even. Let $\xi\in \Xi$ be the unique element such that $\theta$ is a zero of $B_\xi(t)$.
As $B_\xi$ has only simple roots, $B_\xi(t)^2$  divides $P(t)$. As above note that
$P_2=\frac{P(t)}{B_\xi(t)^2}$  has a smaller degree and is non-negative on $S^1$.  We can thus again appeal to the induction hypothesis.

Finally, if $\theta=\pm 1$, then $P$ is divisible by $(t-\theta)^2$, for the same reason. We write $P_2=\frac{P(t)}{B_\theta}$ and by induction we
have $P_2=U_2\ol{U_2}$. Then we put $U=(t-\theta)U_2$.
\end{proof}


\begin{proposition}\label{prop:step3}
Let $A,B\in \L$ be palindromic coprime polynomials.
If for every $z\in S^1_+$, either $A(z)$ or $B(z)$ is positive, then there exist palindromic $P$ and $Q$ in $\L$
such that $PA+QB=1$ and such that $P(z)$ and $Q(z)$ are positive for any $z\in S^1$.
\end{proposition}

The idea behind the proof of Proposition \ref{prop:step3} is that if $a,b$ are real numbers and at least one of them is positive, then we can obviously
find  real numbers $p,q>0$ such that $pa+qb=1$.
The statement of the lemma is that this can be done for palindromic coprime polynomials $A$ and $B$ and  any $z\in S^1$ by palindromic polynomials $P$ and $Q$.

The proof of Proposition \ref{prop:step3} will require the remainder of this section.
Let $A,B\in \L$ be palindromic coprime polynomials such that for every $z\in S^1_+$, either $A(z)$ or $B(z)$ is positive.
Note that $A$ and $B$ being palindromic implies that
 it is also the case that for any  $z\in S^1$ we have that  either $A(z)$ or $B(z)$ is positive.

As $A$ and $B$ are coprime, there exist  $P'$ and $Q'$ in $\L$ such that $P'A+Q'B=1$ by Euclid's algorithm.
We now define $\wti{P}:=\frac12(P'+\ol{P'})$ and $\wti{Q}:=\frac12(Q'+\ol{Q'})$.
Note that $\wti{P}$ and $\wti{Q}$ are palindromic and satisfy the equality $\widetilde{P}A+\widetilde{Q}B=1$
since $A,B$ are palindromic.

The functions $\wti{P}$ and $\wti{Q}$ are not necessarily positive on $S^1$. Our goal is to find a palindromic Laurent polynomial $\gamma$
such that $\wti{P}-\gamma B > 0$ and $\widetilde{Q}+\gamma A >0$ on $S^1$.
To this end, let us define two functions
$\gamma_{\max},\gamma_{\min}\colon S^1\to \R\cup\{\infty,-\infty\}$ as follows:
\begin{equation}\label{eq:cond}
\gamma_{\max}(z)=
\left\{ \ba{ll}
\frac{\wti{P}(z)}{B(z)}&\text{if $B(z)>0$}\\
\infty&\text{if  $B(z)\le 0$}\ea \right. \quad \mbox{ and }
\gamma_{\min}(z)=
\left\{\ba{ll}
\frac{-\wti{Q}(z)}{A(z)}&\text{if $A(z)>0$}\\
-\infty&\text{if $A(z)\le 0$.}\ea \right.
\end{equation}
We also consider the usual ordering on the set $\R\cup \{-\infty,\infty\}$.

\begin{lemma}\label{lem:etamin}
The functions $\gamma_{\min}$ and $\gamma_{\max}$ have the following properties:
\begin{itemize}
\item[(a)] $\gamma_{\min}$ and $\gamma_{\max}$ are symmetric functions on $S^1$.
\item[(b)] Let $z\in S^1$. If $\gamma\in(\gamma_{\min}(z),\gamma_{\max}(z)$, then
\[  \wti{P}(z)-\gamma B(z)> 0 \mbox{ and }\wti{Q}(z)+\gamma A(z)> 0.\]
\item[(c)] For all $z\in S^1$ we have $\gamma_{\min}(z)<\gamma_{\max}(z)$.
\item[(d)] The functions
\[ \ba{rcl} S^1&\to& [-\pi/2,\pi/2]\\
 z&\mapsto &\arctan(\gammamax(z))\mbox{ and }\\
z&\mapsto& \arctan(\gammamin(z))\ea \]
are continuous (here we define $\arctan(\infty)=\pi/2$ and $\arctan(-\infty)=-\pi/2$).
\end{itemize}
\end{lemma}

\begin{proof}
Statement (a) is obvious since $\wti{P}(\ol{z})=\wti{P}(z)$, and the same holds for $A$, $B$ and $Q$, as all these functions are palindromic
polynomials.

We now turn to the proof of (b).
Let $z\in S^1$ and let $\gamma\in(\gamma_{\min}(z),\gamma_{\max}(z))$. First suppose that $A(z)>0$ and $B(z)>0$. Then it follows from the definitions that
\[ \ba{ccccl} \wti{P}(z)-\gamma B(z) &>& \wti{P}(z)-\gammamax(z) B(z)&=&0\\
\wti{Q}(z)+\gamma A(z)&>& \wti{Q}(z)+\gammamin(z) A(z)&=&0.\ea \]
We then suppose that $A(z)= 0$. Note that this implies that $B(z)>0$ by our assumption on $A$ and $B$. In this case
we see that $\wti{P}(z)-\gamma B(z)>0$  as above. Furthermore, we have
\[
\wti{Q}(z)+\gamma A(z) =\wti{Q}(z)=\frac{1}{B(z)}>0.\]
We finally suppose that $A(z)<0$. As above, $B(z)>0$ and $\wti{P}(z)-\gamma B(z)>0$. Furthermore,
\begin{multline*}
\wti{Q}(z)+\gamma A(z)\geq \wti{Q}(z)+\gammamax(z) A(z)=\\
=\wti{Q}(z)+\frac{\wti{P}(z)}{B(z)}A(z)=\frac{\wti{Q}(z)B(z)+\wti{P}(z)A(z)}{B(z)}=\frac{1}{B(z)}>0.
\end{multline*}
Similarly we deal with the case that $B(z)\leq 0$ and $A(z)>0$.
This proves (b).

We then turn to the proof of  (c). It follows from the definitions that  we only have to consider the case that $A(z)>0$ and $B(z)>0$.
In that case we have
\[ \gammamax(z)=\frac{\wti{P}(z)}{B(z)}=-\frac{\wti{Q}(z)}{A(z)}+\frac{1}{A(z)B(z)} >-\frac{\wti{Q}(z)}{A(z)}=\gammamin(z).\]

Finally we turn to the proof of (d). We will first show that $z\mapsto \arctan(\gammamax(z))$  is continuous.
Clearly we only have to show continuity for $z\in S^1$ such that $B(z)=0$.
We will show the following: if $z_i$ is a sequence of points on $S^1$ with $\lim_{i\to\infty} z_i=z$ such that $B(z_i)>0$
for any $i$,
then $\lim_{i\to\infty} \frac{\wti{P}(z_i)}{B(z_i)}=\infty$.
Indeed, since $B(z)=0$ we have $A(z)>0$ by our assumption.
Now $\wti{Q}$ is bounded on $S^1$, in particular from $\wti{P}A+\wti{Q}B=1$ we deduce that
\[ \lim_{i\to\infty} \wti{P}(z_i)=\lim_{i\to\infty} \frac{1-\wti{Q}(z_i)B(z_i)}{A(z_i)}=\frac{1}{A(z)}>0.\]
It now follows that $\lim_{i\to\infty} \frac{\wti{P}(z_i)}{B(z_i)}=\infty$ as desired.

This completes the proof that $z\mapsto \arctan(\gammamax(z))$  is continuous. Similarly one can prove that  $z\mapsto \arctan(\gammamin(z))$
is continuous.
\end{proof}

\begin{lemma}\label{lem:aprox}
There exists a palindromic polynomial $\gamma$ such that $\gammamin(z)< \gamma(z)<  \gammamax(z)$ for any $z\in S^1$.
\end{lemma}

The proof of Lemma~\ref{lem:aprox} might be shortened, but one would have to consider continuous functions with values in $\R\cup\{\pm\infty\}$.
The approach using  $\arctan$ function allows us to avoid such functions.

\begin{proof}
We write $f_1=\arctan(\gammamin(z))$ and $f_2=\arctan(\gammamax(z))$.
By Lemma \ref{lem:etamin} we know that $f_1$ and $f_2$ are continuous functions on $S^1$ with $f_1(z)<f_2(z)$ for all $z\in S^1$.
We can now pick continuous functions $g_1,g_2\co S^1\to (-\pi/2,\pi/2)$ (note the open intervals), such that
$f_1(z)<g_1(z)<g_2(z)<f_2(z)$ for any $z\in S^1$. We can assume that $g_1$ and $g_2$ are symmetric  as $f_1$ and $f_2$ are symmetric  by Lemma~\ref{lem:etamin}(a). We have the inequality
\[ \gammamin(z) <\tan(g_1(z)) < \tan(g_2(z)) <\gammamax(z).\]
We now put
\[c:=\inf\left\{\tan(g_2(z))-\tan(g_1(z))\colon z\in S^1\right\}.\]
We have $c\ge 0$. But since $S^1$ is compact and the functions $\tan g_1$ and $\tan g_2$ are continuous, we have in fact $c>0$.
By Lemma~\ref{lem:converge} we can find a palindromic polynomial $\gamma$ which satisfies
\[ \left| \gamma(z)-\frac{1}{2}\big(\tan(g_2(z))+\tan(g_1(z))\big)\right| < \frac{c}{2}\]
for any $z\in S^1$.
It clearly follows that for any $z\ \in S^1$ we have the desired inequalities
 \[ \gammamin(z)< \gamma(z)< \gammamax(z). \]
\end{proof}

We can now conclude the proof of Proposition  \ref{prop:step3}.
By Lemma \ref{lem:aprox} we can find
a palindromic polynomial $\gamma$ such
that $\gammamin(z)< \gamma(z)< \gammamax(z)$ for all $z\in S^1$.
Then
$P=\widetilde{P}-\gamma B$ and $Q=\widetilde{Q}+\gamma A$ satisfy
$P>0$ and $Q>0$ on $S^1$ by Lemma \ref{lem:etamin} and they satisfy
\be\label{eq:ons1}
P(z)A(z)+Q(z)B(z)=1\text{ for all $z\in S^1$.}
\ee
But both sides of \eqref{eq:ons1} are Laurent polynomials  on $\C\setminus\{0\}$ which agree on infinitely many points. Hence
the equality \eqref{eq:ons1} holds on $\C\setminus\{0\}$. So it must also hold in $\L$.
We have thus finished the proof of Proposition \ref{prop:step3}.

\subsection{The Decomposition Theorem}\label{section:decomp}

In the following recall that  for a palindromic $p=p(t)\in \L$ and any $z\in S^1$ we have $p(z)\in \R$.

\begin{theorem}\label{th:glue}\label{thm:glue}
Assume that $A$ and $B$ are two coprime palindromic Laurent polynomials in $\L$.
Suppose there exists $\varepsilon\in \{-1,1\}$
such that for all $z\in S^1_+$,  at least one of the numbers $\vare A(z)>0$ or $\vare B(z)>0$ is strictly positive, then
\[ \l\bp A&0\\ 0&B\ep \cong \l(\vare AB)\]
as  forms over $\L$.
\end{theorem}

We will prove the theorem by combining  Lemma~\ref{l:step2} and Proposition \ref{prop:step3} with the following lemma.

\begin{lemma}\label{l:step1}
Let $\varepsilon=\pm 1$. If there exist  $U,V\in\L$
such that
\be\label{eq:condonU}
U\overline{U}\cdot A+V\overline{V}\cdot B=\varepsilon,
\ee
then we have
\[ \l\bp A&0\\ 0&B\ep \cong \l(\vare AB).\]
\end{lemma}

\begin{proof}
Suppose that there exist  $U,V\in\L$
which satisfy \eqref{eq:condonU}. Then write $X:=\ol{V}B$ and $Y:=-\ol{U}A$
and  take $N=\bp X&Y\\ U&V\ep$. Note that $\det(N)=\eps$.  Then one calculates that
\[N\bp A&0\\ 0&B\ep\ol{N}^t=\bp \varepsilon AB &0 \\ 0 & \varepsilon\ep.\]
The lemma  follows from  Proposition~\ref{prop:ra81}.
\end{proof}

We can now prove Theorem \ref{th:glue}.

\begin{proof}[Proof of Theorem \ref{th:glue}]
Suppose there exists $\varepsilon\in \{-1,1\}$
such that for all $z\in S^1_+$,  $\vare A(z)>0$ or $\vare B(z)>0$.
By Proposition \ref{prop:step3}  there exist palindromic $P$ and $Q$ in $\L$ such that $PA+QB=\vare$ and such that $P(z)$ and $Q(z)$ are positive for any $z\in S^1$.
By Lemma \ref{l:step2}  there exist $U\in \L$ and $V\in \L$ with $P=U\ol{U}$ and $Q=V\ol{V}$.
The theorem  now follows from Lemma \ref{l:step1}.
\end{proof}

\begin{remark}
One easily sees from Theorem~\ref{th:glue} that if $\xi\in\Xi$ and $|\xi|<1$,
then for any $n\ge 1$ we have  $\l(B_\xi^n)\cong\l(-B_\xi^n)$. On the other hand, if $|\xi|=1$ then $\l(B_\xi^n)$ and $\l(-B_\xi^n)$
are non-isometric. This is a counterpart to the following  fact from the classification of isometric structures (see \cite[Proposition 3.1]{Neu82}
or \cite[Section 2]{Ne95}, compare also \cite{Mi69}): for any $\l\in S^1\setminus\{1\}$ and for any $n\ge 1$ there exist
exactly two distinct isometric structures such that the corresponding monodromy operator is the single Jordan block of size $n$
and eigenvalue $\l$. For any $\l\in\mathbb{C}\setminus \{S^1\cup 0\}$, and any $n\ge 1$, there exists a unique isometric structure such that
the corresponding monodromy operator is a sum of two Jordan blocks of size $n$: one with eigenvalue $\l$ and the other one
with eigenvalue $1/\l$.
\end{remark}

\section{The proof of Theorem \ref{main}}\label{section:mainproof}

After the preparations from the last section we are now in a position to provide the proof of Theorem \ref{main}.

\subsection{Diagonalizing Blanchfield forms}\label{section:diag}

Recall that $\L=\realt$ and $\Omega=\R(t)$.
 We say that a Blanchfield form $\l$  over $\L$ is \emph{diagonalizable} if $\l$ can be represented by a diagonal matrix over
$\L$.
The following is the main result of this section.

\begin{proposition} \label{prop:diag}
Let $\l\colon H\times H\to \Omega/\L$ be a  Blanchfield form over $\L$ such that multiplication by $t\pm 1$ is an isomorphism.
Then  $\l$ is diagonalizable.
\end{proposition}

In order to prove the proposition we will first consider the following special case.

\begin{proposition} \label{prop:diagcycl}
 Let $p\in \L$ be a palindromic polynomial, irreducible over $\R$.
Let $H=\L/p^n\L$ for some $n$ and let  $\l\colon H\times H\to \Omega/\L$ be a  Blanchfield form over $H$.
Then  $\l$ is diagonalizable.
\end{proposition}

In the proof of the proposition we will need the following definition.
If  $g$ is a palindromic polynomial and if $z\in S^1$, then we  say that \emph{ $g$ changes sign at $z$} if in any neighborhood of $z$ on $S^1$ the function $g$ has both positive and negative values.

\begin{proof}
If $p$ is a constant, then there is clearly nothing to prove. Now assume that $p$ is 
 an irreducible palindromic polynomial over $\R$ which is not a constant. It follows that  $\deg(p)=2$
 and we thus deduce from the discussion in Section \ref{section:basicspols} that  the zeros of $p$ lie on $S^1\sm \{\pm 1\}$.
Throughout this proof  let $w$ be the (unique) zero of $p$ which lies in $S^1_+$.
Since $p(1)\ne 0$ we can  multiply $p\in \L$ by the sign of $p(1)$ and we can therefore, without loss of generality,
assume that $p(1)>0$.

\begin{claim}
Let $q$ be a palindromic polynomial coprime to $p$.
Then there exists $g\in \L$ and $\eps\in \{-1,1\}$ such that $q=\eps g\ol{g}\in \L/p^n\L$.
\end{claim}

We first show that the claim implies the proposition.
 Note that $\l$ takes values in $p^{-n}\L/\L$.
 We pick a representative $q'\in \L$ of $p^n\cdot \l(1,1)\in \L/p^n\L$.
 Since $\l$ is hermitian we have $q'\equiv \ol{q}' \bmod p^n\L$.
 We now let $q=\frac{1}{2}(q'+\ol{q}')$. Note that $q$ is palindromic and $q\in \L$ is a
 representative  of $\l(1,1)p^n\in \L/p^n\L$.
Since $\l$ is non-singular it follows that $q$ is coprime to $p$.
 By the claim there exists $g\in \L$ and $\eps\in \{-1,1\}$ such that $q=\eps g\ol{g}\in \L/p^n\L$.
 The map $\L/p^n\L\to \L/p^n\L$ which is given by multiplication by $g$ is easily seen to define an isometry from  $\l(\eps p^n)$ to $\l$.
 (Here recall that $\l(\eps p^n)$ is the Blanchfield form defined by the $1\times 1$-matrix $(\eps p^n)$.)
In particular $\l$ is represented by the $1\times 1$-matrix $\eps p^n$.

We now turn to the proof of the claim.
Given $g\in \L$ we define
\[ s(g):=\# \{ z\in S^1_+ \, |\, g \mbox{ changes sign at }z\}.\]
We will prove the claim by induction on $s(q)$.
If $s(q)=0$, then we denote by $\eps$ the sign of $q(1)$. It follows that $q\eps$ is non-negative
on $S^1$, hence by Lemma \ref{l:step2} there exists $g\in \L$ with $q\eps=g\ol{g}$.

Now suppose the conclusion of the claim holds for any  palindromic $q$ with $s(q)<s$.
Let $q$ be a  palindromic polynomial in $\L$ with $s(q)=s$.
Let $v\in S^1_+$ be a point where $q$ changes sign. Recall that we denote by $w$ the unique zero of $p$ which lies in $S^1_+$.
Note that $v\ne w$ since we assumed that $p$ and $q$ are coprime.

First consider the case that $v<w$. Let $f\in \L$ be an irreducible polynomial such that $f(v)=0$.
Note that $f$ is palindromic and $f(1)\ne 0$. We can thus arrange that $f(1)<0$.
Note that $p$ changes sign on $S^1_+$ precisely at $w$ and $f$ changes sign precisely at $v$.
We thus see that for any $z\in S^1_+$ with $z<w$ we have $p(w)>0$ and for any $z\in S^1_+$ with $z>v$ we have
$f(z)>0$. It follows that for any $z\in S^1_+$ either $f$ or $p$ is positive.
Note that $f$ and $p$ are coprime, we can thus apply Proposition \ref{prop:step3}  to conclude that
there exist palindromic $x$ and $y$ in $\L$
such that $p^nx+fy=1$ and such that $x(z)$ and $y(z)$ are positive for any $z\in S^1$.

We now define $q':=qfy$. Note that
\[ q=q(p^nx+fy)=qfy=q' \in \L/p^n\L.\]
Also note that $q$ and $f$ change sign at $v$. It follows that $qf$ does not change sign at $v$.
Since $z$ is the only zero of $f$ in $S^1_+$ and since $y$ is positive for any $z\in S^1$ it follows that
$s(q')=s(q)-1$. By our induction hypothesis we can thus write
\[ q=q'=\eps g\cdot \ol{g} \in \L/p^n\L\]
for some $g\in \L$.

Now consider the case that $v>w$. Let $f\in \L$ be an irreducible polynomial such that $f(w)=0$ (note that $v$
can not be equal to $-1$, because $q$ changes sign at $v$ and $q$ is palindromic).
Note that $f$ is palindromic and $f(1)\ne 0$. We can thus arrange that $f(1)>0$.
As above we  see that for any point on $S^1_+$ either $f$ or $p$ is negative.
By Proposition \ref{prop:step3}  there exist palindromic $x$ and $y$ in $\L$
such that $p^nx+fy=1$ and such that $x(z)$ and $y(z)$ are negative for any $z\in S^1$.
The proof now proceeds as in the previous case.
\end{proof}

Let $p\in \L$ be an irreducible polynomial. In the following we say that a $\L$-module $H$ is \emph{$p$-primary}
if any $x\in H$ is annihilated by a sufficiently high power of $p$. Given a $p$-primary $\L$-module we introduce the following definitions:
\bn
\item  given $h\in H$ we write $l(h):=\min\{ k\in \N \, |\, p^kh=0\}$,
\item  we write $l(H):=\max\{ l(h)\, |\, h\in H\}$,
\item we  denote by $s(H)$ the minimal number of generators of $H$.
\en
 We will later need the following lemma.

\begin{lemma}\label{lem:subsummand}
Let $p\in \L$ be an irreducible polynomial.
Let $H$ be a finitely generated $p$-primary module. Let $v\in H$ with $l(v)=l(H)$.
Then there exists a direct sum decomposition
\[ H = H'\oplus v\L/p^{l(v)}\]
with $s(H')=s(H)-1$.
\end{lemma}

\begin{proof}
We write $l=l(H)$ and $s=s(H)$.
Since $\L$ is a PID we can apply the classification theorem
for finitely generated $\L$-modules
(see e.g. \cite[Theorems~7.3~and~7.5]{La02})
to find $e_1,\dots,e_k\in \L$ and a submodule $H''\subset H$ with the following properties:
\bn
\item $l(e_i)=l$ for $ i=1,\dots,k$,
\item $H=H''\oplus \bigoplus_{i=1}^k e_i \L/p^l\L$,
\item $s(H'')=s-k$,
\item for any $w\in H''$ we have $l(w)<l$.
\en
Now we can write
$ v=v''+\sum_{i=1}^k a_ie_i$ for some $v''\in H''$ and $a_i\in \L/p^l\L$.
Note that $l(v)=l$ implies that there exists at least one $a_j$ which is coprime to $p$.
We pick $x\in \L$ with $xa_j=1\in \L/p^l$. It is clear that
\[ H=H''\oplus \bigoplus_{i\ne j} e_i \L/p^l\L \oplus xv \L/p^l\L.\]
Since $ xv \L/p^l\L= v \L/p^l\L$ we get the desired decomposition.
Furthermore, it is clear that
\[ s\big(H''\oplus \bigoplus_{i\ne j} e_i \L/p^l\L\big)=s-1.\]
\end{proof}

\begin{lemma} \label{lem:palindromic}
Let $p\in \L$ be a non-zero irreducible palindromic polynomial.
Let $H$ be a $p$-primary module and let $\l\colon H\times H\to \Omega/\L$ be a Blanchfield form.
Then $\l$ is diagonalizable.
\end{lemma}

\begin{proof}
We will prove the lemma by induction on $s(H)$, i.e. on the the number $s$ of generators of $H$. We will use an algorithm, which is a version
of the Gram--Schmidt orthogonalization procedure from linear algebra.
If $s(H)=0$, then clearly there is nothing to prove. So suppose that the conclusion of the lemma holds whenever $s(H)<s$ for some $s>0$.
Now let   $\l\colon H\times H\to \Omega/\L$ be a  Blanchfield form over a $p$-primary module $H$ which  is generated by $s$ elements. We write $l=l(H)$.
Given an element $f$ in the $\L$-module $p^{-n}\L/\L$ we consider again
\[  l(f)=\min\{ k\in \N \, |\, p^kf=0 \in p^{-n}\L/\L\}.\]
We now have the following claim.

\begin{claim}
 There exists $v\in H$ with $l(\l(v,v))=l$.
\end{claim}

To prove the claim, we pick $v\in H$ with $l(v)=l$. It follows from Lemma \ref{lem:subsummand}
that  $v$ generates a subsummand of $H$, in particular we can find a $\L$-homomorphism
$\varphi:H\to p^{-l}\L/\L$ such that $\varphi(v)=p^{-l}\in p^{-l}\L/\L$.
Since $\l$ is non-singular we can find $w\in H$ with $\l(w,v)=p^{-l}\in p^{-l}\L/\L$.
If $l(\l(v,v))=l$ or if $l(\l(w,w))=l$, then we are done.
Otherwise we consider $\l(v+w,v+w)$ which equals
\[  \ba{rcl} \l(v+w,v+w)&=&\l(v,v)+\l(w,w)+\l(v,w)+\ol{\l(v,w)}\\
&=&\l(v,v)+\l(w,w)+p^{-l}+\ol{p}^{-l}\\
&=&\l(v,v)+\l(w,w)+2p^{-l}\in p^{-l}\L/\L.\ea \]
If $l(\l(v,v))<l$ and  if $l(\l(v,w))<l$, then one can now easily see that $l(\l(v+w,v+w))=l$,
i.e.  $v+w$ has the desired property.
This concludes the proof of the claim.

\smallskip
Given the claim, let us pick $v\in H$ with $l(\l(v,v))=l$. This means that  we can write $\l(v,v)=xp^{-l}\L/\L$ for some $x\in \L$ coprime to $p$.
We can in particular find $y\in \L$ such that $yx\equiv 1\bmod p^l$.

By Lemma \ref{lem:subsummand}
 we can find $v_1,\dots,v_{s-1}$ and $l_1,\dots,l_{s-1}$ such that
\[  H= v\L/p^l\L \oplus  \bigoplus_{i=1}^{s-1} v_i \L/p^{l_i}\L.\]
For $i=1,\dots,s-1$ we now define
\[ w_i:=v_i-y\l(v,v_i)v.\]
It follows immediately that $\l(v,w_i)=0\in \Omega/\L$. We thus see that $H$ splits as the orthogonal sum of the submodule generated by $v$ and the submodule generated by $w_1,\dots,w_{s-1}$.

By Proposition \ref{prop:diagcycl} the former is diagonalizable, and by our induction hypothesis the latter is also diagonalizable.
It follows that $\l$ is diagonalizable.

\end{proof}

We are now in a position to prove Proposition \ref{prop:diag}.

\begin{proof}[Proof of Proposition \ref{prop:diag}]
We say  that polynomials $p$ and $q$ in $\L=\realt$ are equivalent, written $p\doteq q\in \realt$,  if they differ by multiplication by a unit in $\realt$, i.e. by an element of the form $rt^i, r\ne 0\in \R$ and $i\in \Z$. Note that if $p\doteq q$, then also $p(t^{-1})\doteq q(t^{-1})$.

We denote by $\PP$ the set of equivalence classes of all non-constant irreducible elements in $\L$ which are not equivalent to $1+t$.
Let   $[p]\in \PP$  with $p(t^{-1})\doteq p(t)$. Since $p$ is irreducible and since we excluded $1+t$
it follows easily that
 $[p]$ is in fact represented by a  palindromic polynomial. We now say that $[p]\in \PP$ is palindromic if it contains a palindromic representative.

Note that $\PP$ inherits an involution $p\mapsto \ol{p}$ coming from the involution on $\L$.
We write $\PP'=\{ p\in \PP \, |\, p \mbox{ palindromic}\}$
and we
define $\PP'':=\{ \{p,\ol{p}\} \, |\, p \mbox{ not palindromic}\}$.
In our notation we will for the most part ignore the distinction between an element $p\in \L$ and the element it represents in $\PP, \PP'$ and $\PP''$.

Given $p\in \PP$ we denote by
\[ H_p:=\{ v\in H \, |\, p^iv=0\mbox{ for some $i\in \N$}\}\]
the $p$-primary part of $H$. Note that $H_{1+t}=0$ by our assumption on $H$.

\begin{claim}
Suppose that $p$ and $\ol{q}$ are non-equivalent irreducible polynomials in $\L$.  Then $\l(a,b)=0$ for any $a\in H_p$ and $b\in H_q$.
\end{claim}

Suppose that $p$ and $\ol{q}$ are not equivalent.
 Since $p$ and $\ol{q}$ are irreducible this means that they are coprime. We can thus find $x,y$ with $xp^l+y\ol{q}=1$, where $l=l(H_p)$.  Let $a\in H_p$ and $b\in H_q$.
Note that multiplication by $\ol{q}$ is an automorphism of $H_p$ with the inverse given by multiplication by $y$.
We can thus write $a=\ol{q}a'$ for some $a'\in H_p$.
We then conclude that
\[ \l(a,b)=\l(\ol{q}a',b)=\l(a,b){q}=\l(a,bq)=0.\]
This concludes the proof of the claim.

Note that by the classification of finitely generated modules over PIDs we get a unique direct sum decomposition
\[ H=\bigoplus\limits_{p\in \PP'} H_p \oplus \bigoplus\limits_{\{p,\ol{p}\}\in \PP''} (H_p\oplus H_{\ol{p}})\]
and it follows from the claim  that this is an orthogonal decomposition.
In particular,  $(H,\l)$ is diagonalizable if the restrictions to
  $H_p$ is diagonalizable for every $p\in \PP'$ and if the restriction of $\l$ to $H_p\oplus H_{\ol{p}}$ is diagonalizable
for every $\{p,\ol{p}\}\in \PP''$.

It follows from Lemma \ref{lem:palindromic} that given $p\in \PP'$ the
 restriction of $\l$ to $H_p$ is diagonalizable.
The following claim thus concludes the proof of the proposition.

\begin{claim}
Let $p\in \L$ be a  non-palindromic irreducible polynomial.
The restriction of $\l$ to $H_p\oplus H_{\ol{p}}$ is diagonalizable.
\end{claim}

First
note that by the first claim of the proof we have $\l(H_p,H_p)=0$ and $\l(H_{\ol{p}},H_{\ol{p}})=0$.
Since $\L$ is a PID we can write $H_p=\oplus_{i=1}^r V_i$ where the $V_i$ are cyclic $\L$-modules.
We then define  $\ol{V}_i$ to be the orthogonal complement in $H_{\ol{p}}$ to $\oplus_{i\ne j}V_j$, i.e.
\[ \ol{V}_i:=\{ w\in H_{\ol{p}}\, |\, \l(w,v)=0\mbox{ for any $v\in \oplus_{i\ne j}V_j$}\}.\]
Since $\l$ is non-singular it follows easily that $H_{\ol{p}}=\oplus_{i=1}^r \ol{V}_i$.
In fact  the decomposition
\[ H_p\oplus H_{\ol{p}} \cong \bigoplus\limits_{i=1}^r (V_i\oplus \ol{V}_i)\]
is an orthogonal decomposition into the $r$ subsummands $V_i\oplus \ol{V}_i$.

It now suffices to show that the restriction of $\l$ to any $V_i\oplus \ol{V}_i$ is diagonalizable.
So let $i\in \{1,\dots,r\}$. Note that $V_i\cong \L/p^n\L$ for some $n$. Let $a$ be a generator of the cyclic $\L$-module $V_i$.
Since $\l$ is non-singular there exists $b\in \ol{V}_i$ such that $\l(a,b)=\ol{p}^{-n}\in \ol{p}^{-n}\L/\L$. Note that $b$ is necessarily a generator of the cyclic $\L$-module $\ol{V}_i$.

Since $p^n$ and $\ol{p}^n$ are coprime we can find $u,v\in \L$ such that $up^n+v\ol{p}^n=1$.
We write $x:=\frac{1}{2}(u+\ol{v})$. Then one can easily verify that $xp^n+\ol{x}\,\ol{p}^n=1$.
Note that it follows in particular that $x$ is coprime to $\ol{p}$.

We now write $w:=a\oplus xb\in V_i\oplus \ol{V}_i$. It is straightforward to see that $\ol{p}^nw$ generates $V_i$
and $p^nw$ generates $\ol{V}_i$, in particular $w$ generates $V_i\oplus \ol{V}_i$.
Furthermore,
\[ \ba{rcl} \l(w,w)&=&\l(a,xb)+\l(xb,a)\\
&=&\l(a,xb)+\ol{\l(a,xb)}\\
&=&x\ol{p}^{-n}+\ol{x}p^{-n}\\
&=&(xp^n+\ol{x}\ol{p}^n)p^{-n}\ol{p}^{-n}\\
&=&p^{-n}\ol{p}^{-n}.\ea \]
This shows that sending $1$ to $w$ defines an isometry from
$\l(p^n\ol{p}^n)$  (i.e. the Blanchfield form defined by the $1\times 1$-matrix $(p^n\ol{p}^n)$) to  the restriction of $\l$ to $V_i\oplus \ol{V}_i$.
\end{proof}

Using Proposition \ref{prop:diag} we can now also prove the following result.

\begin{proposition}
\label{prop:maxoncircle}
Let $\l$ be a  Blanchfield form over $\L$. Let $B=B(t)$ be a hermitian matrix over $\L$ representing $\l$.
Denote by $Z\in S^1_+$ the set of zeros of $\det(B(t))\in \L$.
Then
\[ \mu(\l)=\frac{1}{2}\left(\max\{ \eta_B(z)+\s_B(z) \, |\, z\in Z\} +\max\{ \eta_B(z)-\s_B(z) \, |\, z\in Z\}\right).\]
\end{proposition}

\begin{proof}
By Proposition \ref{prop:diag} there exists a hermitian diagonal matrix $D=\diag(d_1,\dots,d_r)$
over $\L$ with $\l(D)\cong \l(B)$. Recall that $\det(D)\doteq \det(B)$ and $\eta_B(z)=\eta_\l(z)=\eta_D(z)$, $\s_B(z)=\s_\l(z)=\s_D(z)$ for any $z\in S^1$. It thus suffices to prove the claim for $D$.

Given a hermitian matrix $C$ we  write
\[ \Theta_C^{\pm}(z):=\eta_C(z)\pm \s_C(z).\]
Note that
\[ \Theta_D^{\pm}(z)=\sum_{i=1}^r\Theta_{d_i}^\pm(z).\]
It is straightforward to see that for any $i$ the function $\Theta_{d_i}^\pm(z)$ is constant away from the zeros of $d_i$ and that the values at a zero are relative maxima. The proposition now follows immediately.
\end{proof}

\subsection{Elementary diagonal forms}\label{section:basicdiagonal}

We say that a matrix is \emph{elementary diagonal} if it is of the form
\[E=\diag(e_1,\dots,e_M),\]
where for $k=1,\dots,M$ we have $e_k=\varepsilon_kB_{\xi_k}^{n_k}$ for some $\varepsilon_k\in \{-1,1\}$, $n_k\in \N$ and $\xi_k\in \Xi$.

\begin{lemma}~\label{lem:splitting}
Let $D$ be a hermitian matrix over $\L$  such that $\det(D(\pm 1))$ is non-zero. Then
there exists an elementary diagonal matrix $E$ such that $\l(D)\cong \l(E)$.
\end{lemma}

\begin{proof}
Let $D$ be a hermitian matrix over $\L$  such that $\det(D(\pm 1))\ne 0$.
By Proposition~\ref{prop:diag} we can without loss of generality  assume that $D$ is a diagonal $n\times n$-matrix.

We will use an inductive argument. We denote by $d_k$ the  $k$-th entry on the diagonal of $D$.  Since $d_k$ is a real
polynomial and since $\ol{D}=D$, we have a unique decomposition
\[d_k(t)=\varepsilon_kc_k\prod_{\xi} B_\xi(t)^{n_{k,\xi}},\]
with $\varepsilon_k\in \{-1,1\}, c_k\in \R_{>0}$ and
where $\xi$ runs over all elements $\Xi$,
and where $n_{k,\xi}$ is zero for all but finitely many $\xi$.

We write $C=\diag(\sqrt{c_1},\dots,\sqrt{c_n})$. After replacing $D$ by $C^{-1}D(C^{-1})^t$ we can assume that $c_i=1$ for all $i$.

 Assume now that there exists a $\xi$ with
$|\xi|<1$ such that $n_{k,\xi}>0$. Let us define
\[A(t)= B_\xi(t)^{n_{k,\xi}}\mbox{ and }  B(t)=\varepsilon_k\frac{d_{k}(t)}{A(t)}.\]
Note that $A(1)>0$. Since $A(t)$ has no zeros on $S^1$ we have in fact that $A(z)>0$ for any $z\in S^1$.
We can therefore use Theorem~\ref{th:glue} to show that the matrix
$\diag(d_1,d_2,\dots,d_k,\dots,d_n)$ is congruent to $\diag(d_1,\dots,d_{k-1},A,B,d_{k+1},\dots,d_n)$. In this way
we can split off all terms with $|\xi|<1$.

It remains to consider the case when $d_k(t)=\varepsilon_k\prod_{\xi\in S^1_+}B_\xi(t)^{n_{k,\xi}}$.
Let $\xi\in S^1_+$ be the minimal number in $S^1_+$ with $n_{k,\xi}>0$. We now define
\[A(t)= \varepsilon_k B_{\xi_1}(t)^{n_{k,\xi}}\mbox{ and } B(t)=\varepsilon_k(-1)^{n_{k,\xi}}
\prod_{\xi'\ne \xi} B_{\xi'}(t)^{n_{k,\xi'}}.\]
Note that for $z\in S^1_+$ with $z>\xi$ we have $\sign(A(z))=\varepsilon_k(-1)^{n_{k,\xi}}$
and for $z\in S^1_+$ with $z\leq \xi$ we have $\sign(B(z))=\varepsilon_k(-1)^{n_{k,\xi}}$.
It thus follows from Theorem~\ref{th:glue} that the matrices $(d_k)$ and $\bp A&0\\0&B\ep$ give
rise to the same Blanchfield form.
The lemma now follows from a straightforward induction argument.
\end{proof}

\subsection{Conclusion of the proof of Theorem \ref{main}}\label{section:conclusion}

It follows from the discussion in Sections \ref{section:class} and \ref{section:statementmainthm}
that the following theorem is equivalent to Theorem \ref{main}.

\begin{theorem}\label{th:realequal}
Let $B$ be a hermitian  matrix over $\L$ such that $\det(B(\pm 1))\ne 0$.
Then
\[ \nr(\l(B))=\max\{\mu(B),\eta(B)\}.\]
\end{theorem}

We will first prove  two   special cases of Theorem~\ref{th:realequal}.

\begin{proposition}\label{prop:oncircle}
Let $B$ be a hermitian matrix over $\L$ such that all zeros of $\det(B)\in \L$ lie on $S^1\sm \{\pm 1\}$.
Then $n_\R(B)=\mu(B)$.
\end{proposition}

Note that if  $\det(B)\in \L$ has no zero outside of the unit circle then it can also be seen directly
that $\eta(B)\leq \mu(B)$.

\begin{proof}
By  Lemma~\ref{lem:splitting} it suffices to prove the proposition for an elementary diagonal matrix
of the form $E=\diag(e_1,\dots,e_M)$, where for $k=1,\dots,M$ we have $e_k=\varepsilon_kB_{\xi_k}^{n_k}$ for some $\varepsilon_k\in \{-1,1\}$, $n_k\in \N$ and $\xi_k\in \Xi\cap S^1=S^1_+$. By Lemma \ref{lem:easy} it remains to show that $\mu(B)\geq n_\R(B)$.
This will be achieved by proving the following claim.

\begin{claim}
Let $E=\diag(e_1,\dots,e_M)$ be such an elementary diagonal matrix. We write $s=\mu(E)$.
Then there exists a  decomposition
\[\{1,\dots,M\}=\bigcup_{a=1}^{s}I_a\]
into pairwise disjoint sets, and for each $a=1,\dots,s$ there exists  $\kappa_a\in \{-1,1\}$ such that
\[\l(E)\cong\lambda\big(\diag\big(\kappa_1 \prod_{i\in I_1} e_i, \, \dots,\,  \kappa_s \prod_{i\in I_{s}} e_i\big)\big).\]
\end{claim}

We will prove the claim  by induction on the size $M$ of the elementary diagonal matrix. The case  $M=0$ is trivial.
So now suppose that the statement of the claim holds whenever the size of the elementary diagonal matrix is at most $M-1$.
Let $E=\diag(e_1,\dots,e_{M})$ be   an elementary diagonal matrix  such that  $\xi_k\in S^1\cap \Xi\subset S^1_+$ for $k=1,\dots,M$.
Without loss of generality we can assume  that $\xi_1\le\dots\le \xi_{M}$ on $S^1_+$.

We now write $E':=\diag(e_1,\dots,e_{M-1})$. We write $s:=\mu(E)$ and $s':=\mu(E')$. We then apply our induction hypothesis to $E'$.
We obtain  the corresponding decomposition
$\{1,\dots,M-1\}=I'_1\cup\dots \cup I'_{s'}$ and signs $\kappa_1',\dots,\kappa'_{s'}$.  For $a=1,\dots,s'$, let
\[\rho_a=\kappa'_a \prod_{i\in I'_{a}} e_i.\]
In the following we  write $\varepsilon=\varepsilon_M, n=n_M, e=e_M$ and $\xi_M=\xi$.

\emph{Case 1.} First suppose there exists an  $a\in \{1,\dots,s'\}$ such that  $\rho_a(\xi)\ne 0$ and such that $\sign(\rho_a(\xi))=\varepsilon$.
Note that
\[ \sign(B_{\xi}^{n}(z))=\sign(B_{\xi}^{n}(1))=\sign(\varepsilon)
\]  for any $z\in [1,\xi)\subset S^1_+$ since $B_{\xi}^{n}$ has no zeros on  $z\in [1,\xi)$.
Now recall that we assumed that  $\xi_1\le\dots\le \xi_{M}=\xi$ on $S^1_+$. It follows that $\rho_a$ has no zeros on $[\xi,-1]\subset S^1_+$.
It thus follows that
\[ \sign(\rho_a(z))=\sign(\rho_a(\xi))=\varepsilon \]
for any $z\in [\xi,-1]$.  We can thus apply Theorem~\ref{th:glue}
to conclude that
\be \label{equ:l} \l(\varepsilon \rho_a\cdot e)\cong  \l(\diag(\rho_a,e)).\ee
We will now prove the following claim.

\begin{claim}
$s=s'$.
\end{claim}

Note that (\ref{equ:l}) implies that $\l(E)$ can be represented by an $s'\times s'$-matrix, in particular it follows that $s\leq s'$.
We will now show that $s\geq s'$.
Given a hermitian matrix $C$ over $\ct$ and $z\in S^1$ we  write
\[ \Theta_C^{\pm}(z):=\eta_C(z)\pm \s_C(z).\]
By  Proposition \ref{prop:maxoncircle} we have
\[ \ba{rcl}\mu(E')&=&
\frac{1}{2}\left(\max\{ \Theta^+_{E'}(z) \, |\, z\in S^1_+\} + \max\{ \Theta^-_{E'}(z) \, |\, z\in S^1_+\}\right)\\[2mm]
&=&
\frac{1}{2}\left(\max\{ \Theta^+_{E'}(z) \, |\, z\in [1,\xi]\} + \max\{ \Theta^-_{E'}(z) \, |\, z\in [1,\xi]\}\right) \\[2mm]
&=&
\frac{1}{2}\left(\max\{ \Theta^+_{E'}(z)+\varepsilon \, |\, z\in [1,\xi]\} + \max\{ \Theta^-_{E'}(z)-\varepsilon \, |\, z\in [1,\xi]\}\right).\ea \]
Note that $\Theta^\pm_{E}(z)=\Theta^\pm_{E'}(z)+\Theta^\pm_{e}(z)$.
It is straightforward to verify that $ \Theta^{\pm}_{e}(z)\mp \varepsilon$
is greater or equal than zero for any $z\in [0,\xi]$. We thus conclude that
\[ \Theta_{E'}^\pm \pm \varepsilon=\Theta_{E}^\pm-(\Theta_{\varepsilon}^\pm \mp \varepsilon )\leq \Theta_E^\pm \]
on $S^1_+$.
It follows that
\[ \ba{rcl}\mu(E')&\leq &
\frac{1}{2}\left(\max\{ \Theta^+_{E}(z) \, |\, z\in [1,\xi]\} + \max\{ \Theta^-_{E}(z) \, |\, z\in [1,\xi]\}\right) \\[2mm]
&\leq &\frac{1}{2}\left(\max\{ \Theta^+_{E}(z) \, |\, z\in S^1_+\} + \max\{ \Theta^-_{E}(z) \, |\, z\in S^1_+\} \right) \\
&=&\mu(E).\ea \] This concludes the proof that $s=s'$.
We now define $I_{a}=I'_{a}\cup\{M\}$ and  $I_b=I_b'$ for $b\neq a$ and the induction step is proved for Case~1.

\emph{Case 2.} Now suppose that for any  $a\in \{1,\dots,s'\}$ we either have  $\rho_a(\xi)= 0$ or  $\sign(\rho_a(\xi))=-\varepsilon$.
We claim that $s=s'+1$.
We write $R:=\diag(\rho_1,\dots,\rho_{s'},e)$.
We can thus represent $E$ by the matrix $R$  of size $s'+1$. It follows that $s\leq s'+1$.
We now write $k:=\# \mbox\{ a \in \{1,\dots,s'\} \, |\, \rho_a(\xi)=0\}$.
We have
\[ \ba{rcl}\mu(E')&=&\mu(R) \\
&\geq & \frac{1}{2}\max\{ \eta_R(z)+\varepsilon \s_R(z) \, |\, z\in S^1\} \\[2mm]
&\geq & \frac{1}{2}\left( \eta_R(\xi)+\varepsilon\s_R(\xi)\right) \\[2mm]
&=&(k+1)+(s'-k) =s'+1.\ea \]
We now take $I_a:=I'_a$ for $a\in \{1,\dots,s'\}$ and we define $I_{s'+1}=\{M\}$.
\end{proof}

%

We now consider the next special case of Theorem \ref{main}.

\begin{proposition}\label{prop:notoncircle}
Let $B$ be a hermitian matrix over $\L$ such that $\det(B)\in \L$ has no zero on the unit circle.
Then $n_\R(B)=\eta(B)$.
\end{proposition}

Note that if $\det(B)\in \L$ has no zero on the unit circle, then
$\eta_B$ and $\sigma_B$ are constant functions on the unit circle, hence $\mu(B)=0$.

\begin{proof}
As in the proof of Proposition \ref{prop:oncircle} we only have to consider the case that $B$ is
an elementary diagonal matrix $B=\diag(e_1,\dots,e_M)$. Note that the zeros of $e_1,\dots,e_M$ do not lie on $S^1$.
 We write $s=\eta(B)$.
 Also, as we pointed out before, in light of Lemma \ref{lem:easy} it suffices to prove that  $s=\eta(B)\geq \nr(B)$.
 
 It is straightforward to see that one can decompose $\{1,\dots,M\}$ into subsets $I_1,\dots,I_{s}$ with the following property:
given $k,l\in I_b$ with  $k\neq l$ the polynomials $e_k$ and $e_l$ have different roots.
It is clear that one can find such  $I_1,\dots,I_{s}$, since for any
$\xi\not\in S^1$ there exist at most $s$ indices $k\in \{1,\dots,M\}$ for which $e_k$ has root at $\xi$.

Since the sign of any product of product of the $e_i$ is constant on the unit circle we can now apply Theorem \ref{thm:glue} repeatedly to show that there exist $\eps_b\in \{-1,1\}$ such that
\[  \l(B)\cong\lambda\big(\diag\big(\eps_1 \prod_{j\in I_1} e_j, \, \dots,\,  \eps_s \prod_{j\in I_{s}} e_j\big)\big).\]
We thus showed that $\nr(B)\leq s=\eta(B)$. 
\end{proof}

We are now ready to finally provide a proof of Theorem~\ref{th:realequal}.

\begin{proof}[Proof of Theorem~\ref{th:realequal}]
Let $B$ be a square matrix over $\L$ such that $B(1)$ and $B(-1)$ are non-degenerate.
We write $s_i:=s_i(B)$.
It follows from Lemma \ref{lem:splitting} together with the proofs of Propositions \ref{prop:oncircle} and \ref{prop:notoncircle}
that there exist palindromic
$f_1,\dots,f_{\mu}\in \L$ with no zeros outside of $S^1$ and palindromic $g_1,\dots,g_{\eta}\in \L$ with no zeros on $S^1$ such that
\[ \l(B)\cong \l(\diag(f_1,\dots,f_{\mu},g_1,\dots,g_{\eta})).\]
Note that the  sign of any $g_i$ is constant on the unit circle.
It follows from Theorem \ref{thm:glue} that for any $k\in \{1,\dots,\min(\mu,\eta)\}$ we have $\l(\diag(f_k,g_k))\cong \l((\varepsilon_kf_kg_k))$ for some  $\varepsilon_k\in \{\pm 1\}$. This shows, that if $\mu\ge \eta$,
\[\l(B)\cong\l(\diag(\varepsilon_1f_1g_1,\dots,\varepsilon_{\eta}f_{\eta}g_{\eta},f_{\eta+1},\dots f_{\mu})),\]
while, if $\eta>\mu$
\[\l(B)\cong\l(\diag(\varepsilon_1f_1g_1,\dots,\varepsilon_{\mu}f_{\mu}g_{\mu},g_{\mu+1},\dots g_{\eta})).\]
We thus showed that  $\max\{\mu(B),\eta(B)\}\geq \nr(B)$.
Together with  Lemma \ref{lem:easy} we now obtain the desired equality  $\max\{\mu(B),\eta(B)\}= \nr(B)$.
\end{proof} 

We point out that the proof of Theorem \ref{th:realequal} in fact provides a proof of the following slightly more precise statement.

\begin{theorem}
Let $B$ be a hermitian  matrix over $\L$ such that  $\det(B(\pm 1))\ne 0$. Let $s=\max\{\mu(B),\eta(B)\}$.
Then there exists a diagonal hermitian $s\times s$-matrix $D$ over $\L$ such that $\l(D)\cong \l(B)$.
\end{theorem}

\section{Examples}
We will first summarize a few properties of the invariant $\eta$ before proceeding with various examples.

\subsection{Basic properties of $\eta$}

We have the following result.
\begin{lemma}
For any knot $K$, the following numbers are equal.
\begin{itemize}
\item[(a)] The maximum of nullities $\eta(K)$.
\item[(b)] The real Nakanishi index, i.e. the minimal number of generators of the $\realt$ module $H_1(X(K);\realt)$.
\item[(c)] The rational Nakanishi index, i.e. the minimal number of generators of the $\qt$ module $H_1(X(K),\qt)$.
\item[(d)] The maximal index $k$, for which the $k-$th Alexander polynomial $\Delta_k$ is not $1$
\item[(e)] The bigger of the two following numbers
\[\max_{\l\colon 0<|\l|<1}\sum_{k=1}^\infty q^k_\l\textrm{ and }\max_{|\l|=1}\sum_{k=1}^\infty\sum_{u=\pm 1}p^k_\l(u),\]
where the numbers $p^k_\lambda(u)$ and $q^k_\lambda$ are
the Hodge numbers defined in \cite{BN13}.
\end{itemize}
\end{lemma}
\begin{proof}
The fact that (a), (b), (c) and (d) are equal is well known to the experts. For the convenience of the reader we point
out that (a)=(e) follows from \cite[Lemma 4.4.6]{BN13}, (c)=(d)=(e) is \cite[Proposition 4.3.4]{BN13}. It is obvious that (b)$\le$(c)
and (d)$\le$(b).
\end{proof}

We now recall, that give a knot $K$ we denote by $-K$  the knot which is given by reversing the orientation and taking the mirror image.
It is well-known that $\eta(-K)=\eta(K)$ and $\mu(-K)=-\mu(K)$. 
Also note that the invariant $\mu$ is additive under connect sum, in particular we see that $\mu(K\# -K)=0$ for any knot $K$.

From the lemma it follows that for a knot $K$ and signs $\eps_1,\dots,\eps_n\in \{-1,1\}$
we have
\[\eta(\eps_1 K\#\dots\# \eps_n K)=n\cdot\eta(K).\]
On the other hand, given two knots $K_1$ and $K_2$ we have
\[\eta(K_1\# K_2)\in\{\max(\eta(K_1),\eta(K_2)),\dots,\eta(K_1)+\eta(K_2)\}.\]
Finally, if the Alexander polynomials of $K_1$ and $K_2$ are coprime, then $\eta(K_1\# K_2)=\max(\eta(K_1),\eta(K_2))$.

\subsection{Some concrete examples} 

We first consider the knot $K=12a_{896}$.
Its Alexander polynomial is
\[ \Delta_K(t)=2-11t+26t^2-40t^3+45t^4-40t^5+26t^6-11t^7+2t^8.\]
The Alexander polynomial has no multiple roots, we therefore see that $\eta(K)=1$.

The graph of the function $x\to \s(e^{2\pi ix})$
is presented on Figure~\ref{fig:12a896}.
The maximum of the Levine-Tristram signature is $2$, the minimum is $-2$. All the jumps of the Levine--Tristram signatures
correspond to single roots of the Alexander polynomial. We thus see that  $\mu(K)=2$.
Note that $\mu(K)=2$  is bigger than half the maximum of
the absolute value of the Levine--Tristram signature function.

\begin{figure}
\begin{pspicture}(-5,-2)(5,2)
\psline{->}(-4,0)(4,0)
\psline{->}(-3,-1.5)(-3,1.5)
\psline(-2.9,0.1)(-3.1,-0.1)\rput(-3.3,-0.2){\psscalebox{0.8}{$0$}}
\psline(-3.2,0.6)(-2.8,0.6)\rput(-3.4,0.6){\psscalebox{0.8}{$1$}}
\psline(-3.2,1.2)(-2.8,1.2)\rput(-3.4,1.2){\psscalebox{0.8}{$2$}}
\psline(-3.2,-0.6)(-2.8,-0.6)\rput(-3.5,-0.6){\psscalebox{0.8}{$-1$}}
\psline(-3.2,-1.2)(-2.8,-1.2)\rput(-3.5,-1.2){\psscalebox{0.8}{$-2$}}
\psline[linecolor=blue,linewidth=1.2pt](-3,0)(-2.35,0)(-2.35,1.2)(-2.15,1.2)(-2.15,0)(-1.4,0)(-1.4,-1.2)(1.4,-1.2)(1.4,0)(2.15,0)(2.15,1.2)(2.35,1.2)(2.35,0)(3,0)
\pscircle[fillcolor=red,fillstyle=solid,linestyle=none](-2.35,0.6){0.08}
\pscircle[fillcolor=red,fillstyle=solid,linestyle=none](-2.15,0.6){0.08}
\pscircle[fillcolor=red,fillstyle=solid,linestyle=none](-1.4,-0.6){0.08}
\pscircle[fillcolor=red,fillstyle=solid,linestyle=none](2.35,0.6){0.08}
\pscircle[fillcolor=red,fillstyle=solid,linestyle=none](2.15,0.6){0.08}
\pscircle[fillcolor=red,fillstyle=solid,linestyle=none](1.4,-0.6){0.08}
\psline(3,-0.1)(3,0.1)\rput(3,-0.3){\psscalebox{0.8}{$1$}}
\psline(0,-0.1)(0,0.1)\rput(0,-0.3){\psscalebox{0.8}{$1/2$}}
\end{pspicture}
\caption{Graph of the signature function of the knot $12a_{896}$, more precisely the function $x\to\s(e^{2\pi ix})$.
The jumps of the signature function occur at the places, corresponding to roots of the Alexander polynomial $(2-3t+2t^2)(1-4t+6t^2-7t^3+6t^4-4t^5+1)$ on
the unit circle. Numerically they are $x\sim 0.115$, $x\sim 0.12149$, $x\sim 0.2697$, $x\sim 0.7302$, $x\sim 0.8785$, $x\sim 0.8850$. The graph is taken
from \cite{CL11}.}\label{fig:12a896}.
\end{figure}
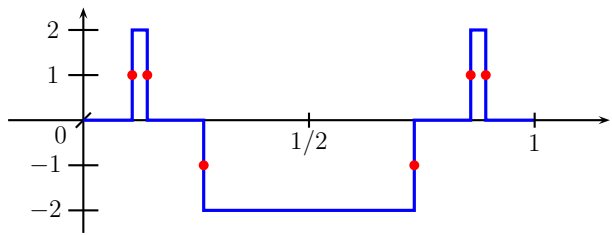

We refer to the authors' webpage \cite{BF12web} for more information on unknotting numbers for knots
 with up to 12 crossings.
\medskip

We list now some examples, which are built from connected sums of different knots.

\begin{enumerate}
\item For any knot $K$ with non-trivial Alexander polynomial and $\eta=1$ (for example we could take $K$ to be the trefoil), the knot $K'=K\#-K$ has $\mu=0$ and $\eta=2$. The connected
sum of $n$ copies of $K'$ has $\mu=0$ but $\eta=n$ can be arbitrarily large.
\item The torus knots $T_{2,2k+1}$ have signature $2k$, the span of signatures is $\mu(T_{2,2k+1})=k$ but $\eta=1$. This example and the example above
show that $\mu$ and $\eta$ are, in general, completely independent.
\item For any torus knot $T_{2,2k+1}$ we saw in (2) that  $n_\R=k$, but for $T_{2,2k+1}\#-T_{2,2k+1}$ we have $\mu=0$, $\eta=2$, so $n_\R=2$.
 The Blanchfield Form dimension $n_\R$ is therefore not additive. Note that this is in contrast to the conjecture (see \cite[Problem 1.69(B)]{Ki97}) that the unknotting number is additive under connect sum.
\item The knots $6_2$ and $10_{32}$ have both $n_\R=1$ (see \cite{CL11} for graphs of their signature functions). But their sum
$6_2\#10_{32}$ also has $n_\R=1$. 
Therefore, $n_\R(K_1\#K_2)$ can be equal to $1$ even if
$n_\R(K_1)=n_\R(K_2)=1$.
\end{enumerate}

Finally note that in \cite[Theorem 18]{Li11} Livingston uses the Levine--Tristram signature function to
define a new invariant $\rho(K)$ which gives a lower bound on the $4$-genus and in particular on the unknotting number.
Livingston furthermore shows that $\rho(-5_1\#10_{132})=3$, whereas $n_\R(K)=2$.
This shows that the Blanchfield Form dimension $n_\R(K)$ is not the optimal unknotting
information, which can be obtained from Levine--Tristram signatures and nullities.

On the other hand there are many examples for which $\rho(K)=0$, e.g. for all knots with vanishing Levine-Tristram signature function, but for which $\eta(K)>0$.
This shows that $\rho(K)$ and $n_\R(K)$ are independent lower bounds on the unknotting number.

We conclude this paper with the following question:

\begin{question}
What is the optimal lower bound on the unknotting number that can be obtained using  Levine--Tristram signatures and nullities?
\end{question}

\end{document}